\documentclass[11pt]{article}
\usepackage{enumerate}
\usepackage{amssymb,a4wide,latexsym,makeidx,epsfig,fleqn}
\usepackage{amsthm}
\usepackage{amsmath}
\usepackage{enumerate}
\usepackage{extarrows}
\usepackage{graphicx}
\usepackage{subfigure}
\usepackage{float}
\usepackage[justification=centering]{caption}
\newtheorem{theorem}{Theorem}[section]

\newtheorem{definition}[theorem]{Definition}
\newtheorem{lemma}[theorem]{Lemma}

\newtheorem{corollary}[theorem]{Corollary}

\begin{document}
\textwidth 150mm \textheight 225mm
\title{On the distance spectral radius, fractional matching and factors of graphs with given minimum degree
\thanks{Supported by the National Natural Science Foundation of China (Nos. 12001434 and 12271439) and the Natural Science Basic Research Program of Shaanxi Province (No. 2022JM-006 and 2024JC-YBQN-0015).}}
\author{Zengzhao Xu$^{a}$, Weige Xi$^{a}$\footnote{Corresponding author.}, Ligong Wang$^{b}$\\
{\small $^{a}$ College of Science, Northwest A\&F University, Yangling, Shaanxi 712100, P.R. China}\\
{\small $^{b}$ School of Mathematics and Statistics, Northwestern Polytechnical University,}\\
{\small  Xi'an, Shaanxi 710129, P.R. China}\\
{\small E-mail: xuzz0130@163.com; xiyanxwg@163.com; lgwangmath@163.com}\\}

\date{}
\maketitle
\begin{center}
\begin{minipage}{120mm}
\vskip 0.3cm
\begin{center}
{\small {\bf Abstract}}
\end{center}
{\small  A fractional matching of $G$ is a function $f: E(G)\to [0,1]$ such that $\sum_{e\in E_G(v_i)}f(e)\le 1$ for any $v_i\in V(G)$, where $E_G(v_i)=\{e: e\in E(G) \ \textrm{and}\  e \ \textrm{is incident with} \ v_i\}$. Let $\alpha_f(G)$ denote the fractional matching number of $G$, which is defined as $\alpha_f(G)=\max\{\sum_{e\in E(G)}f(e): f\ \textrm{is a fractional matching of} \ G\}$.
Let $\{G_1,G_2,G_3,\dots\}$ be a set of graphs, a $\{G_1,G_2,G_3,\dots\}$-factor of a graph $G$ is a spanning subgraph of $G$ such that each component of which is isomorphic to one of $\{G_1,G_2,G_3,\dots\}$. In this paper, we first establish a sharp upper bound for the distance spectral radius to guarantee that $\alpha_f(G)>\frac{n-k}{2}$ in a graph $G$ of order $n$ with given minimum degree, where $0<k<n$ is an integer. Then we give a sharp upper bound on the distance spectral radius of a graph $G$ with given minimum degree $\delta$ to ensure that $G$ has a $\{K_2, \{C_k\}\}$-factor, where $3\le k<+\infty$ is an integer. Moreover, we obtain a sharp upper bound on the distance spectral radius for the existence of a $\{K_{1,1},K_{1,2},\dots,K_{1,k}\}$-factor with $2\le k<+\infty$ in a graph $G$ with given minimum degree.

}
\end{minipage}
\end{center}

\section{Introduction }

Throughout this paper all graphs considered are simple, connected and undirected. A graph $G$ is denoted by $G=(V(G),E(G))$, where $V(G)=\{v_1,v_2,\ldots, v_n\}$
is the vertex set and $E(G)$ is the edge set. The order of $G$ is $|V(G)|=n$ and its size is $|E(G)|=m$. The set of neighbours of the vertex $v_i$ denoted by $N_G(v_i)$, is the set of vertices adjacent to $v_i$. The degree $d(v_i)$ of the vertex $v_i$ in $G$ is the number of vertices of $G$ adjacent to $v_i$, i.e. $d(v_i)=|N_G(v_i)|$. We use $\Delta$ and $\delta$ to denote the maximum degree and the minimum degree of $G$, respectively. Let $C_n$, $K_n$ and $K_{1,n-1}$ denote the cycle, the complete graph and the star of order $n$, respectively. For $S \subseteq V(G)$, we use $G-S$ to denote the subgraph obtained from $G$ by deleting the vertices in $S$ together with their incident edges. For an edge $e\in E$, we use $G-e$
to denote the subgraph obtained from $G$ by deleting the edge $e$. For two vertex disjoint graphs $G_1$ and $G_2$, we use $G_1+ G_2$ to denote the disjoint union of $G_1$ and $G_2$. The join $G_1\vee G_2$ of $G_1$ and $G_2$ is the graph obtained from $G_1+G_2$ by adding all possible edges between $V(G_1)$ and $V(G_2)$. For more details and concepts, the readers may refer to \cite{BM}.

A fractional matching of a graph $G$ is a function $f$ giving each edge a number in $[0,1]$ such that $\sum_{e\in E_G(v_i)}f(e)\le 1$ for each $v_i\in V(G)$, where $E_G(v_i)=\{e:e\in E(G) \ \textrm{and}\  e \ \textrm{ is incident}$ $\textrm{with} \ v_i\}$. Let $\alpha_f(G)$ denote the fractional matching number of $G$, which is defined as $\alpha_f(G)=\max\{\sum_{e\in E(G)}f(e): f\ \textrm{is a fractional matching of} \ G\}$. Summing the inequality constraints for all vertices has $2\sum_{e\in E(G)}f(e)\leq n$, thus $\alpha_f(G)\leq \frac{n}{2}$. Therefore, a fractional matching of $G$ is called a fractional perfect matching if $\sum_{e\in E(G)}f(e)=\frac{n}{2}$.

For a connected graph $G$ of order $n$, the distance $d_{ij}$ between vertices $v_i$ and $v_j$ is the length of a shortest path from $v_i$ to $v_j$ in $G$.
The distance matrix $D(G)=(d_{ij})_{n\times n}$ of $G$ is a square matrix of order $n$. The largest eigenvalue of $D(G)$, denoted by $\mu(G)$, is called the distance spectral radius of $G$.

Recently, the relationship between the fractional matching number and the eigenvalues of graphs has been investigated by several researchers. For example, O \cite{O} determined the connections between the spectral radius of an $n$-vertex connected graph with minimum degree and its fractional matching number, and they also gave a lower bound on the fractional matching number in terms of the spectral radius and minimum degree. Xue et al. \cite{XZS} discussed the relations between the fractional matching number and the Laplacian spectral radius of a graph, and they obtained some lower bounds on the fractional matching number. Pan et al. \cite{PLZ} studied the existence of fractional perfect matchings for graphs with given order and minimum degree in terms of the signless Laplacian spectral radius. Yan et al. \cite{YLS} gave some lower bounds of distance Laplacian spectral radii of $n$-vertex graphs in terms of fractional matching number. Li et al. \cite{LMZ} established an upper bound on the distance spectral radius of a graph $G$ to ensure that $G$ has a fractional perfect matching. Lou et al. \cite{LLA} obtained a tight lower bound of the spectral radius to guarantee the fractional matching number more than $\frac{n-k}{2}$ in a graph with minimum degree $\delta$, where $0<k<n$ is an integer. Some classical conclusions about fractional matching can refer to \cite{HLZZ,ZHR}

Motivated by \cite{LLA}, in this paper we firstly investigate the relations between the distance spectral radius of a graph and its fractional matching number.

For a set $\{G_1,G_2,G_3,\dots\}$ of graphs, a $\{G_1,G_2,G_3,\dots\}$-factor of a graph $G$ is a spanning subgraph of $G$ each component of which is isomorphic to one of $\{G_1,G_2,G_3,\dots\}$. Hence if $H$ is a $\{G_1,G_2,G_3,\dots\}$-factor of graph $G$, then $H$ is a subgraph of $G$ such that $V(H)=V(G)$ and each component of $H$ is contained in $\{G_1,G_2,G_3,\dots\}$. Specifically, the $\{K_{1,1},K_{1,2},\dots,K_{1,k}\}$-factor is also called a star factor of $G$.

The factors of graphs have received a lot of attention of researchers and their theories are well developed in recent years. In \cite{T}, Tutte presented a sufficient and necessary condition for the existence of a $\{K_2, \{C_k\}\}$-factor, where $3\le k<+\infty$ is an integer. In \cite{AK}, Amahashi and Kano gave a necessary and sufficient condition of the existence of a $\{K_{1,1},K_{1,2},\dots,K_{1,k}\}$-factor in a graph, where $2\le k<+\infty$ is an integer. In \cite{LLA}, Lou et al. provided a tight spectral radius condition for the existence of a $\{K_{1,1},K_{1,2},\dots,K_{1,k}\}$-factor with $2\le k<+\infty$ in a graph with minimum degree $\delta$.  In \cite{ML}, Miao and Li established a lower bound on the size of a graph $G$ to guarantee that $G$ contains a star factor. They also determined an upper bound on the spectral radius of $G$ to ensure that $G$ has a star factor.

Inspired by \cite{LLA}, in this paper we provide conditions in terms of distance spectral radius for the existence of some factors.

The rest of this paper is divided into the following sections. In Section 2, we will introduce some concepts and lemmas to prove the theorems in the following sections. In Section 3, we establish a sharp upper bound for the distance spectral radius to guarantee $\alpha_f(G)>\frac{n-k}{2}$ in a graph $G$ of order $n$ with given minimum degree $\delta$, where $0<k<n$ is an integer. In Section 4, we give a sharp upper bound on the distance spectral radius of a graph $G$ with given minimum degree $\delta$ to ensure that $G$ has a $\{K_2, \{C_k\}\}$-factor, where $3\le k<+\infty$ is an integer. In addition, we also obtain a sharp upper bound on the distance spectral radius for the existence of a $\{K_{1,1},K_{1,2},\dots,K_{1,k}\}$-factor with $2\le k<+\infty$ in a graph $G$ with given minimum degree $\delta$.

\section{Preliminaries}

In this section, we will introduce some preliminary results and useful lemmas which will be used in the follows. Firstly, we give some lemmas about fractional matching number, the sufficient and necessary condition for the existence of $\{K_2,\{C_k\}\}$-factor and $\{K_{1,1},K_{1,2},\dots,K_{1,k}\}$-factor in a graph. For $S \subseteq V(G)$, we use $i(G-S)$ to denote the number of isolated vertices in the graph $G-S$.

\begin{lemma}\label{le:2} (\cite{SU}) \ Let $G$ be a graph of order $n$. Then
	$$\alpha_f(G)=\frac{1}{2}(n-\max\{i(G-S)-|S|: \ S \subseteq V(G)\}).$$
\end{lemma}

\begin{lemma}\label{le:3} (\cite{T}) \ Let $G$ be a graph of order $n$ and $3\le k<+\infty$ be an integer. Then $G$ has a $\{K_2,\{C_k\}\}$-factor if and only if $i(G-S)\le |S|$ for every $S\subseteq V(G)$.
\end{lemma}

\begin{lemma}\label{le:4} (\cite{AK}) \ Let $G$ be a graph of order $n$ and $2\le k<+\infty$ be an integer. Then $G$ has a $\{K_{1,1},K_{1,2},\dots,K_{1,k}\}$-factor if and only if $i(G-S)\le k|S|$ for every $S\subseteq V(G)$.
\end{lemma}

Nonnegative matrix theory plays an important role in the study of spectral graph theory, especially in comparing the eigenvalues of graphs. Next we will introduce some concepts and lemmas of nonnegative matrices.

\begin{definition}\label{def1}(\cite{AR}) Let $A=(a_{ij})$ and $B=(b_{ij})$ be two $n\times n$ matrices. If $a_{ij}\le b_{ij}$ for all $i$ and $j$, then $A\le B$. If $A\le B$ and $A\ne B$, then $A<B$. If $a_{ij}< b_{ij}$ for all $i$ and $j$, then $A\ll B$.
\end{definition}

\begin{lemma}\label{le:6} (\cite{AR}) \ Let $A$ and $B$ be two $n \times n$ nonnegative matrices with the spectral radius $\rho(A)$ and $\rho(B)$, respectively. If $A\le B$, then $\rho(A) \le \rho(B)$. In addition, if $A< B$ and $B$ is irreducible, then $\rho(A)<\rho(B)$.
\end{lemma}

\begin{lemma}\label{le:7} (\cite{AR}) \ Let $m<n$, $A$ and $B$ be the $n \times n$ and $m \times m$ nonnegative matrices with the spectral radius $\rho(A)$ and $\rho(B)$, respectively. If $B$ is a principal submatrix of $A$, then $\rho(B) \le \rho(A)$. In addition, if $A$ is irreducible, then $\rho(B)<\rho(A)$.
\end{lemma}

By Lemma \ref{le:6}, we can get an important result about the distance radius of a connected graph $G$.

\begin{lemma}\label{le:8} (\cite{G}) \ Let $e$ be an edge of graph $G$ such that $G-e$ is still connected. Then
	$\mu(G)<\mu(G-e)$.
\end{lemma}

Finally, we give the definition of the equitable quotient matrix and its property.

\noindent\begin{definition}\label{D1}(\cite{YYSX}) Let $M$ be a complex matrix of order $n$ described in the following block form
	\begin{equation*}
		M=\begin{bmatrix}
			M_{11} & \cdots & M_{1t} \\
			\vdots& \ddots & \vdots \\
			M_{t1} &\cdots & M_{tt}
		\end{bmatrix},
	\end{equation*}
where the blocks $M_{ij}$ are the $n_i\times n_j$ matrices for any $1\le i,j\le t$ and $n=n_1+n_2+\cdots+n_t$. For  $1\le i,j\le t$, let $b_{ij}$ denote the average row sum of $M_{ij}$, i.e. $b_{ij}$ is the sum of all entries in $M_{ij}$ divided by the number of rows. Then $B(M)=(b_{ij})$(or simply $B$) is called the quotient matrix of $M$. If, in addition, for each pair $i$, $j$, $M_{ij}$ has a constant row sum, then $B$ is called the equitable quotient matrix of $M$.
\end{definition}

\noindent\begin{lemma}\label{le:10}(\cite{YYSX}) Let $B$ be the equitable quotient matrix of $M$, where $M$ is as shown in Definition \ref{D1}. In addition, let $M$ be a nonnegative matrix. Then the spectral radius relation satisfies $\rho(B)=\rho(M)$, where $\rho(B)$ and $\rho(M)$ denote the spectral radii of $B$ and $M$ respectively.
\end{lemma}

\section{Fractional matching number and distance spectral radius}

First of all, we will introduce a useful theorem which was proposed by Scheinerman and Ullman \cite{SU}. This theorem plays an important role in the fractional matching of graphs.

\noindent\begin{theorem}\label{le:1} (\cite{SU}) \ Let $G$ be a graph of order $n$. Then \\
	(a) any fractional matching satisfies $\alpha_f(G) \le \frac{n}{2}$.\\
	(b) $2\alpha_f(G) $ is an integer.
\end{theorem}

Since $\alpha_f(G) \le \frac{n}{2}$, it is natural to consider that can we find a condition for the distance spectral radius that makes the fractional matching number of a graph $G$ more than $\frac{n-k}{2}$ with given minimum degree, where $0<k<n$ is an integer? In addition, can we characterize the corresponding spectral extremal graphs? Based on the above considerations, we give the following theorem.

\noindent\begin{theorem}\label{T1.2}  \ Let $G$ be a connected graph of order $n\ge9k+10\delta+2$ with the minimum degree $\delta$, where $0<k<n$ is an integer. If $\mu(G)\le \mu(K_{\delta} \vee (K_{n-2\delta-k}+(\delta+k)K_1))$, then $\alpha_f(G)>\frac{n-k}{2}$ unless $G\cong K_{\delta} \vee (K_{n-2\delta-k}+(\delta+k)K_1)$. (see Figure 1)
\end{theorem}

\begin{figure}[H]
	\begin{centering}
		\includegraphics[scale=1]{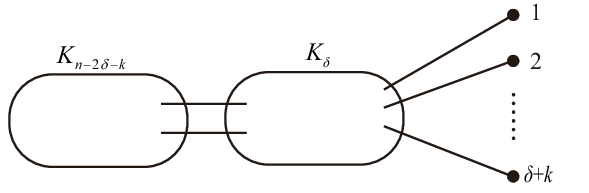}
		\caption{The extremal graph of Theorem 3.2.}\label{Fig.1.}
	\end{centering}
\end{figure}

\begin{proof}  Let $G$ be a graph of order $n\ge9k+10\delta+2$ and minimum degree $\delta$. By the way of contradiction, we assume that the fractional matching number $\alpha_f(G)\le\frac{n-k}{2}$. By Lemma \ref{le:2}, there exists a vertex subset $S\subseteq V(G)$ such that $\alpha_f(G)=\frac{1}{2}(n-(i(G-S)-|S|))$. Thus $i(G-S)-|S|\ge k$. The set of isolated vertices in $G-S$ is denoted by $T$. Let $|S|=s$ and $|T|=t$. Then $t=i(G-S)\ge s+k$. Since $s+t\le n$, we have $s\le\frac{n-k}{2}$. It is easy to see that $N_G(T)\subseteq S$, thus $s\ge \delta$. Let $G_1=K_s \vee (K_{n-2s-k}+(s+k)K_1)$. Obviously, $G$ is a spanning subgraph of $G_1$. By Lemma \ref{le:8}, we have
	\begin{equation}\label{eq:1}
		\mu(G)\ge \mu(G_1),
	\end{equation}
with the equality holds if and only if $G\cong G_1$. Since $\delta \le s\le \frac{n-k}{2}$, we will discuss the proof into two cases according to the value of $s$.
	
	{\bf Case 1.} $s=\delta$.
	
	In this case, we konw that $G_1=K_{\delta} \vee (K_{n-2\delta-k}+(\delta+k)K_1)$. By (\ref{eq:1}), we have
	$$\mu(G)\ge\mu(G_1)=\mu(K_{\delta} \vee (K_{n-2\delta-k}+(\delta+k)K_1)).$$
What's more, according to the assumed condition $\mu(G)\le \mu(K_{\delta} \vee (K_{n-2\delta-k}+(\delta+k)K_1))$ and $G$ is a spanning subgraph of $K_{\delta} \vee (K_{n-2\delta-k}+(\delta+k)K_1)$, we conclude that $G\cong K_{\delta} \vee (K_{n-2\delta-k}+(\delta+k)K_1)$. Obviously, deleting vertices in $K_{n-2\delta-k}$ and $(\delta+k)K_1$ does not cause a graph $G$ to become disconnected. Hence,
	$$\max\{i(G-S)-|S|:\forall S \subseteq V(G)\}=i(G-V(K_{\delta}))-|V(K_{\delta})|=k.$$
Thus, with the help of Lemma \ref{le:2}, it can be concluded that $\alpha_f(G)=\frac{n-k}{2}$.
Hence, if $s=\delta$, then $G\cong K_{\delta} \vee (K_{n-2\delta-k}+(\delta+k)K_1)$.
	
	{\bf Case 2.} $\delta<s\le \frac{n-k}{2}.$
	
	Let $G_1=K_{s} \vee (K_{n-2s-k}+(s+k)K_1)$. We divide $V(G_1)$ into three parts: $V(K_s)$, $V(K_{n-2s-k})$ and $V((s+k)K_1)$. Then the distance matrix of $G_1$, denoted by $D(G_1)$, is
	\begin{equation*}
		\begin{bmatrix}
			(J-I)_{s\times s} & J_{s\times (n-2s-k)} & J_{s\times (s+k)} \\
			J_{(n-2s-k)\times s}& (J-I)_{(n-2s-k)\times (n-2s-k)} & 2J_{(n-2s-k)\times (s+k)}\\
			J_{(s+k)\times s}& {2J}_{(s+k)\times (n-2s-k)} & 2(J-I)_{(s+k)\times (s+k)}
		\end{bmatrix},
	\end{equation*}
where $J_{n\times m}$ denotes the $n\times m$ all-one matrix and $I_{n\times n}$ denotes the $n\times n$ identity square matrix. Then the equitable quotient matrix of the distance matrix $D(G_1)$, denoted by $M_s$, for the partition $V(K_s)\cup V(K_{n-2s-k})\cup V((s+k)K_1)$ is
	\begin{equation*}
		M_s=\begin{bmatrix}
			s-1 & n-2s-k & s+k \\
			s & n-2s-k-1 & 2(s+k) \\
			s & 2(n-2s-k) & 2(s+k-1)
		\end{bmatrix},
	\end{equation*}
and the characteristic polynomial of $M_s$ is
\begin{align*}
f_s(x)&=x^3+(-s-n-k+4)x^2+(5s^2-2ns+7ks-s-2kn-3n+2k^2-k+5)x\\
	& \ \ \ -2s^3+(n-3k+5)s^2+(kn-2n-k^2+7k)s-2kn-2n+2k^2+2.
\end{align*}	
Let $\lambda_1(M_s)$ denote the largest real root of the equation $f_s(x)=0$. By Lemma \ref{le:10}, we have $\mu(G_1)=\lambda_1(M_s)$. Obviously, by replacing $s$ with $\delta$, we can get the equitable quotient matrix $M_{\delta}$ of $G_2=K_{\delta} \vee (K_{n-2\delta-k}+(\delta+k)K_1)$. Similarly, we can get the characteristic polynomial $f_{\delta}(x)$ of $M_{\delta}$ and $\mu(G_2)=\lambda_1(M_{\delta})$ is the largest real root of the equation $f_{\delta}(x)=0$. Then we have
	\begin{align*}
f_s(x)-f_{\delta}(x)&=(\delta-s)[x^2+(2n+1-5(\delta+s)-7k)x+(2s-n+2\delta+3k-5)s\\
&\ \ \ \ +(2\delta-n+3k-5)\delta+2n-kn+k^2-7k].
\end{align*}	
Since $G_1$ and $G_2$ are spanning subgraphs of $K_n$, by Lemma \ref{le:8}, $\mu(G_1)>\mu(K_{n})=n-1$ and $\mu(G_2)>\mu(K_{n})=n-1$. Then we will prove that $f_s(x)-f_{\delta}(x)<0$ for $x\in[n-1,+\infty)$. Since $\delta<s$, we only need to show $g(x)>0$ for $x\in[n-1,+\infty)$, where
	$$g(x)=x^2+(2n+1-5(\delta+s)-7k)x+(2s-n+2\delta+3k-5)s+(2\delta-n+3k-5)\delta+2n-kn+k^2-7k.$$
Since $s\leq\frac{n-k}{2}$, the symmetry axis of $g(x)$ is
	\begin{align*}
		\tilde{x}&=\frac{5(\delta+s)+7k-2n-1}{2}\\
		&=\frac{5}{2}s+\frac{5}{2}\delta+\frac{7}{2}k-n-\frac{1}{2}\\
		&\le\frac{5}{4}(n-k)+\frac{5}{2}\delta+\frac{7}{2}k-n-\frac{1}{2}\\
		&=\frac{1}{4}n+\frac{5}{2}\delta+\frac{9}{4}k-\frac{1}{2}.
	\end{align*}
	
Note that $n\ge9k+10\delta+2>3k+\frac{10}{3}\delta+\frac{2}{3}$ and $n>3k+\frac{10}{3}\delta+\frac{2}{3}\iff \frac{1}{4}n+\frac{5}{2}\delta+\frac{9}{4}k-\frac{1}{2}<n-1$, we have $\frac{5(\delta+s)+7k-2n-1}{2}<n-1$, which implies $g(x)$ is increasing with respect to $x\in[n-1,+\infty)$. Then
	$$g(x)\ge g(n-1)=2s^2+(3k+2\delta-6n)s+3n^2+2\delta^2+(3k-6n)\delta-n-8kn+k^2.$$
Let
$$h(s)=g(n-1)=2s^2+(3k+2\delta-6n)s+3n^2+2\delta^2+(3k-6n)\delta-n-8kn+k^2.$$
Recall that $\delta<s\le \frac{n-k}{2}$, then
	\begin{align*}
		\frac{dh}{ds}&=4s+3k+2\delta-6n\\
		&\le2n-2k+3k+2\delta-6n\\
		&=-4n+k+2\delta<0.
	\end{align*}
Hence $h(s)$ is decreasing with respect to $s\in[\delta+1,\frac{n-k}{2}]$. By a calculation,
\begin{align*}
	h(s)\ge h(\frac{n-k}{2})&=\frac{1}{2}[n^2-(10\delta+9k+2)n+4\delta^2+4k\delta]\\
	&>\frac{1}{2}[n^2-(10\delta+9k+2)n].
\end{align*}
Since $n\ge9k+10\delta+2$, we have $h(s)>0$ and $g(x)\ge g(n-1)=h(s)>0$, which implies $f_s(x)<f_{\delta}(x)$ for $x\in[n-1,+\infty)$. Since $\min\{\mu(G_1),\mu(G_2)\}>n-1$, it can be concluded that $\mu(G_1)>\mu(G_2)$. Thus
$$\mu(G)\ge\mu(G_1)>\mu(G_2)=\mu(K_{\delta} \vee (K_{n-2\delta-k}+(\delta+k)K_1),$$
a contradiction. This completes the proof.
\end{proof}

Let $k=1$, we can get a distance spectral radius condition in a graph $G$ of order $n$ with minimum degree $\delta$ such that $\alpha_f(G)>\frac{n-1}{2}$. However, by Theorem \ref{le:1}, we have $\alpha_f(G)\le\frac{n}{2}$ and $2\alpha_f(G) $ is an integer. Thus, we can obtain the following corollary about the fractional perfect matching based on the distance spectral radius.

\noindent\begin{corollary}\label{C1.3}  \ Let $G$ be a connected graph of order $n\ge11+10\delta$ with minimum degree $\delta$. If $\mu(G)\le \mu(K_{\delta} \vee (K_{n-2\delta-1}+(\delta+1)K_1))$, then $G$ contains a fractional perfect matching unless $G\cong K_{\delta} \vee (K_{n-2\delta-1}+(\delta+1)K_1).$
\end{corollary}

\section{$\{K_2, \{C_k\}\}$-factor, $\{K_{1,1},K_{1,2},\dots,K_{1,k}\}$-factor and distance spectral radius}

\quad\quad Tutte \cite{T} gave a sufficient and necessary condition for a graph contains a $\{K_2, \{C_k\}\}$-factor with $3\le k<+\infty$. In this section, we firstly obtain a distance spectral radius condition that makes a graph $G$ with given minimum degree $\delta$ has a $\{K_2, \{C_k\}\}$-factor.

\noindent\begin{theorem}\label{T1.4}  \ Let $G$ be a connected graph of order $n\ge11+10\delta$ with minimum degree $\delta$, where $3\le k<+\infty$ is an integer. If $\mu(G)\le \mu(K_{\delta} \vee (K_{n-2\delta-1}+(\delta+1)K_1))$, then $G$ has a $\{K_2,\{C_k\}\}$-factor unless $G\cong K_{\delta} \vee (K_{n-2\delta-1}+(\delta+1)K_1)$. (see Figure 2)
\end{theorem}
\begin{figure}[H]
	\begin{centering}
		\includegraphics[scale=1]{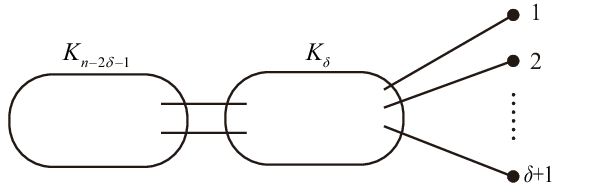}
		\caption{The extremal graph of Theorem 4.1.}\label{Fig.2.}
	\end{centering}
\end{figure}
\begin{proof}  Let $G$ be a graph of order $n\ge11+10\delta$ and minimum degree $\delta$. By the way of contradiction, we assume that $G$ has no a $\{K_2,\{C_k\}\}$-factor. Then by Lemma \ref{le:3}, there
exists a vertex subset $S\subseteq V(G)$ such that $i(G-S)-|S|\ge 1$. The set of isolated vertices in $G-S$ denoted by $T$. Let $|S|=s$ and $|T|=t$. Then we have $t=i(G-S)\ge s+1$. Since $s+t\le n$, we have $s\le\frac{n-1}{2}$. It is easy to find that $N_G(T)\subseteq S$, thus $s\ge \delta$. Let $G_1=K_s \vee (K_{n-2s-1}+(s+1)K_1)$. Obviously, $G$ is a spanning subgraph of $G_1$. By Lemma \ref{le:8}, we have
\begin{equation}\label{eq2}
	\mu(G)\ge \mu(G_1),
\end{equation}
where equality holds if and only if $G\cong G_1$. Recall that $\delta \le s\le \frac{n-1}{2}$, we will divide the proof into two cases according to the value of $s$.

{\bf Case 1.} $s=\delta$.

In this case, we have $G_1=K_{\delta} \vee (K_{n-2\delta-1}+(\delta+1)K_1)$. By (\ref{eq2}), it can be concluded that
$$\mu(G)\ge\mu(G_1)=\mu(K_{\delta} \vee (K_{n-2\delta-1}+(\delta+1)K_1)).$$
Moreover, according to the assumed condition $\mu(G)\le \mu(K_{\delta} \vee (K_{n-2\delta-1}+(\delta+1)K_1))$ and $G$ is a spanning subgraph of $K_{\delta} \vee (K_{n-2\delta-1}+(\delta+1)K_1)$, we conclude that $G\cong K_{\delta} \vee (K_{n-2\delta-1}+(\delta+1)K_1)$. Note that
the vertices of $(\delta+1)K_1$ are only adjacent to $\delta$ vertices of $K_{\delta}$. Then $K_{\delta} \vee (K_{n-2\delta-1}+(\delta+1)K_1)$ does not contain a spanning subgraph where each component is a member of $\{K_2,\{C_k\}\}$. Hence $K_{\delta} \vee (K_{n-2\delta-1}+(\delta+1)K_1)$ does not have a $\{K_2,\{C_k\}\}$-factor.

{\bf Case 2.} $\delta<s\le \frac{n-1}{2}$.

Let $G_1=K_{s} \vee (K_{n-2s-1}+(s+1)K_1$. We divide $V(G_1)$ into three parts: $V(K_s)$, $V(K_{n-2s-1})$ and $V((s+1)K_1)$. We use $M_s$ to denote the equitable quotient matrix of the distance matrix $D(G_1)$ for the partition $V(K_s)\cup V(K_{n-2s-1})\cup V((s+1)K_1)$ of $V(G_1)$. Then
\begin{equation*}
	M_s=\begin{bmatrix}
		s-1 & n-2s-1 & s+1 \\
		s & n-2s-2 & 2(s+1) \\
		s & 2(n-2s-1) & 2s
	\end{bmatrix}.
\end{equation*}
By direct calculation, the characteristic polynomial of $M_s$ is
$$f_s(x)=x^3+(3-s-n)x^2+(5s^2-2ns+6s-5n+6)x-2s^3+(n+2)s^2+(6-n)s-4n+4.$$
Let $\lambda_1(M_s)$ denote the largest real root of the equation $f_s(x)=0$. By Lemma \ref{le:10}, we have $\mu(G_1)=\lambda_1(M_s)$. Obviously, by replacing $s$ with $\delta$, we can get the equitable quotient matrix $M_{\delta}$ of $G_2=K_{\delta} \vee (K_{n-2\delta-1}+(\delta+1)K_1)$. Similarly, we can get the characteristic polynomial $f_{\delta}(x)$ of $M_{\delta}$ and $\mu(G_2)=\lambda_1(M_{\delta})$ is the largest real root of the equation $f_{\delta}(x)=0$. Then
$$f_s(x)-f_{\delta}(x)=(\delta-s)[x^2+(2n-5(\delta+s)-6)x+(2s-n+2\delta-2)s+(2\delta-n-2)\delta+n-6].$$
Since $G_1$ and $G_2$ are spanning subgraphs of $K_n$, by Lemma \ref{le:8}, we obtain $\mu(G_1)>\mu(K_{n})=n-1$ and $\mu(G_2)>\mu(K_{n})=n-1$.
Then we will prove that $f_s(x)-f_{\delta}(x)<0$ for $x\in[n-1,+\infty)$. Since $\delta<s$, we only need to prove $g(x)>0$ for $x\in[n-1,+\infty)$, where
$$g(x)=x^2+(2n-5(\delta+s)-6)x+(2s-n+2\delta-2)s+(2\delta-n-2)\delta+n-6.$$
Note that $s\le \frac{n-1}{2}$, the symmetry axis of $g(x)$ is
\begin{align*}
	\tilde{x}&=\frac{5(\delta+s)+6-2n}{2}\\
	&=\frac{5}{2}s+\frac{5}{2}\delta+\frac{7}{2}-n-\frac{1}{2}\\
	&\le\frac{5}{4}(n-1)+\frac{5}{2}\delta+\frac{7}{2}-n-\frac{1}{2}\\
	&=\frac{1}{4}n+\frac{5}{2}\delta+\frac{7}{4}.
\end{align*}
Recall that $n\ge11+10\delta>\frac{10}{3}\delta+\frac{11}{3}$ and $n>\frac{10}{3}\delta+\frac{11}{3}\iff \frac{1}{4}n+\frac{5}{2}\delta+\frac{7}{4}<n-1$, we have $\frac{5(\delta+s)+6-2n}{2}<n-1$, which implies $g(x)$ is increasing with respect to $x\in[n-1,+\infty)$. Then
$$g(x)\ge g(n-1)=2s^2+(3+2\delta-6n)s+3n^2+2\delta^2+(3-6n)\delta-9n+1.$$
Let
$$h(s)=g(n-1)=2s^2+(3+2\delta-6n)s+3n^2+2\delta^2+(3-6n)\delta-9n+1.$$
Note that $\delta<s\le \frac{n-1}{2}$, we have
\begin{align*}
	\frac{dh}{ds}&=4s+3+2\delta-6n\\
	&\le2n-2+3+2\delta-6n\\
	&=-4n+1+2\delta<0.
\end{align*}
Hence $h(s)$ is decreasing with respect to $s\in[\delta+1,\frac{n-1}{2}]$. Then
\begin{align*}
	h(s)\ge h(\frac{n-1}{2})&=\frac{1}{2}[n^2-(10\delta+11)n+4\delta^2+4\delta]\\
	&>\frac{1}{2}[n^2-(10\delta+11)n].
\end{align*}
Since $n\ge11+10\delta$, we have $h(s)>0$. Hence $g(x)\ge g(n-1)=h(s)>0$, which implies $f_s(x)<f_{\delta}(x)$ for  $x\in[n-1,+\infty)$. Since $\min\{\mu(G_1),\mu(G_2)\}>n-1$, it can be concluded that $\mu(G_1)>\mu(G_2)$. Thus
$$\mu(G)\ge\mu(G_1)>\mu(G_2)=\mu(K_{\delta} \vee (K_{n-2\delta-1}+(\delta+1)K_1),$$
a contradiction. This completes the proof.
\end{proof}

In the following, we will present a distance spectral radius condition to guarantee the existence of $\{K_{1,1},K_{1,2},\dots,K_{1,k}\}$-factor in a graph $G$ with given minimum degree $\delta$.

\noindent\begin{theorem}\label{T1.5}  \ Let $G$ be a connected graph of order $n\ge\frac{3+5k}{k^2}+3+(\frac{3}{k}+5+2k)\delta$ with minimum degree $\delta$, where $2\le k<+\infty$ is an integer. If $\mu(G)\le \mu(K_{\delta} \vee (K_{n-k\delta-\delta-1}+(k\delta+1)K_1))$, then $G$ has a $\{K_{1,1},K_{1,2},\dots,K_{1,k}\}$-factor unless $G\cong K_{\delta} \vee (K_{n-k\delta-\delta-1}+(k\delta+1)K_1)$. (see Figure 3)
\end{theorem}

\begin{figure}[H]
	\begin{centering}
		\includegraphics[scale=1]{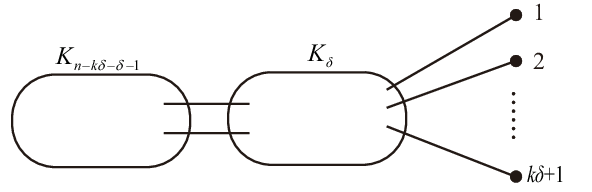}
		\caption{The extremal graph of Theorem 4.2.}\label{Fig.3.}
	\end{centering}
\end{figure}

\begin{proof}  Let $G$ be a graph of order $n\ge\frac{3+5k}{k^2}+3+(\frac{3}{k}+5+2k)\delta$ and minimum degree $\delta$. By the way of contradiction, we assume that $G$ has no a $\{K_{1,1},K_{1,2},\dots,K_{1,k}\}$-factor. Then by Lemma \ref{le:4}, there exists a vertex subset $S\subseteq V(G)$ such that $i(G-S)-k|S|\ge 1$. The set of isolated vertices in $G-S$ denoted by $T$. Let $|S|=s$ and $|T|=t$. Then $t=i(G-S)\ge ks+1$. Since $s+t\le n$, we have $s\le\frac{n-1}{k+1}$. It is easy to find that $N_G(T)\subseteq S$, thus $s\ge \delta$. Let $G_1=K_s \vee (K_{n-ks-s-1}+(ks+1)K_1)$. Obviously, $G$ is a spanning subgraph of $G_1$. By Lemma \ref{le:8}, we obtain
	\begin{equation}\label{eq:2}
		\mu(G)\ge \mu(G_1),
	\end{equation}
with the equality holds if and only if $G\cong G_1$. Since $\delta \le s\le \frac{n-1}{k+1}$, we will discuss the proof in two ways according to the value of $s$.
	
	{\bf Case 1.} $s=\delta$.
	
	In this case, we have $G_1=K_{\delta} \vee (K_{n-k\delta-\delta-1}+(k\delta+1)K_1)$. By (\ref{eq:2}), it can be concluded that
	$$\mu(G)\ge\mu(G_1)=\mu(K_{\delta} \vee (K_{n-k\delta-\delta-1}+(k\delta+1)K_1)).$$
Furthermore, according to the assumed condition $\mu(G)\le \mu(K_{\delta} \vee (K_{n-k\delta-\delta-1}+(k\delta+1)K_1))$ and $G$ is a spanning subgraph of $K_{\delta} \vee (K_{n-k\delta-\delta-1}+(k\delta+1)K_1)$, we conclude that $G\cong K_{\delta} \vee (K_{n-k\delta-\delta-1}+(k\delta+1)K_1)$. Note that
the vertices of $(k\delta+1)K_1$ are only adjacent to $\delta$ vertices of $K_{\delta}$. Then $K_{\delta} \vee (K_{n-k\delta-\delta-1}+(k\delta+1)K_1)$ does not contain a spanning subgraph where each component is a member of $\{K_{1,1},K_{1,2},\dots,K_{1,k}\}$. This implies that $K_{\delta} \vee (K_{n-k\delta-\delta-1}+(k\delta+1)K_1)$ does not have a $\{K_{1,1},K_{1,2},\dots,K_{1,k}\}$-factor.
	
	{\bf Case 2.} $\delta<s\le \frac{n-1}{k+1}.$
	
	Let $G_1=K_s \vee (K_{n-ks-s-1}+(ks+1)K_1)$. We divide $V(G_1)$ into three parts: $V(K_s)$, $V(K_{n-ks-s-1})$, $V((ks+1)K_1)$. Then the equitable quotient matrix of $D(G_1)$, denoted by $M_s$, for the
	partition $V(K_s)\cup V(K_{n-ks-s-1})\cup V((ks+1)K_1)$ of $V(G_1)$ is
	\begin{equation*}
		M_s=\begin{bmatrix}
			s-1 & n-ks-s-1 & ks+1 \\
			s & n-ks-s-2 & 2(ks+1) \\
			s & 2(n-ks-s-1) & 2ks
		\end{bmatrix},
	\end{equation*}
and the characteristic polynomial of $M_s$ is
\begin{align*}
f_s(x)&=x^3+(-ks-n+3)x^2+(2k^2s^2+3ks^2-2kns+3ks+3s-5n+6)x-(k^2+k)s^3\\
	&\ \ \ \ +(kn+2k^2+k-1)s^2+(n-2kn+4k+2)s+4-4n.
\end{align*}
Let $\lambda_1(M_s)$ denote the largest real root of the equation $f_s(x)=0$. By Lemma \ref{le:10}, we have $\mu(G_1)=\lambda_1(M_s)$. Obviously, by replacing $s$ with $\delta$, we can get the equitable quotient matrix $M_{\delta}$ of $G_2=K_{\delta} \vee (K_{n-k\delta-\delta-1}+(k\delta+1)K_1)$. Similarly, we can get the characteristic polynomial $f_{\delta}(x)$ of $M_{\delta}$ and $\mu(G_2)=\lambda_1(M_{\delta})$ is the largest real root of the equation $f_{\delta}(x)=0$. Then
\begin{align*}
f_s(x)-f_{\delta}(x)&=(\delta-s)[kx^2+(-2k^2s-3ks+2kn-2k^2\delta-3k\delta-3k-3)x\\
&\ \ \ +(k^2s+ks-kn+k^2\delta+k\delta-2k^2-k+1)s\\
&\ \ \ +(-kn+k^2\delta+k\delta-2k^2-k+1)\delta-n+2kn-4k-2].
\end{align*}	
Since $G_1$ and $G_2$ are spanning subgraphs of $K_n$, by Lemma \ref{le:8}, we have $\mu(G_1)>\mu(K_{n})=n-1$ and $\mu(G_2)>\mu(K_{n})=n-1$. Then we will prove that $f_s(x)-f_{\delta}(x)<0$ for $x\in[n-1,+\infty)$. Since $\delta<s$, we only need to prove $g(x)>0$ for $x\in[n-1,+\infty)$, where
\begin{align*}
g(x)&=kx^2+(-2k^2s-3ks+2kn-2k^2\delta-3k\delta-3k-3)x\\
&\ \ \ +(k^2s+ks-kn+k^2\delta+k\delta-2k^2-k+1)s\\
& \ \ \  +(-kn+k^2\delta+k\delta-2k^2-k+1)\delta-n+2kn-4k-2.
\end{align*}
Note that $s\le \frac{n-1}{k+1}$ and $k\geq2$, the symmetry axis of $g(x)$ is
	\begin{align*}
	\tilde{x}&=-\frac{(-2k^2s-3ks+2kn-2k^2\delta-3k\delta-3k-3)}{2k}\\
	&=\frac{3}{2k}+\frac{3}{2}+(\frac{3}{2}+k)\delta+\frac{3}{2}s+ks-n\\
	&\le\frac{9}{4}+(\frac{3}{2}+k)\delta+(\frac{3}{2}+k)s-n\\
	&\le\frac{9}{4}+(\frac{3}{2}+k)\delta+(\frac{3}{2}+k)\frac{n-1}{k+1}-n\\
	&=\frac{5}{4}+(\frac{3}{2}+k)\delta+\frac{1}{2(k+1)}n-\frac{1}{2(k+1)}.
\end{align*}
Since $n\ge\frac{3+5k}{k^2}+3+(\frac{3}{k}+5+2k)\delta>\frac{1}{2k+1}[\frac{9}{2}(k+1)-1]+(\frac{3}{2}+k)\frac{2(k+1)}{2k+1}\delta$ and $n>\frac{1}{2k+1}[\frac{9}{2}(k+1)-1]+(\frac{3}{2}+k)\frac{2(k+1)}{2k+1}\delta\iff \frac{5}{4}+(\frac{3}{2}+k)\delta+\frac{1}{2(k+1)}n-\frac{1}{2(k+1)}<n-1$, we have $-\frac{(-2k^2s-3ks+2kn-2k^2\delta-3k\delta-3k-3)}{2k}<n-1$, which implies $g(x)$ is increasing with respect to $x\in[n-1,+\infty)$. Then
	\begin{align*}
		g(x)\ge g(n-1)&=(k^2+k)s^2+(1+2k+k\delta+k^2\delta-4kn-2k^2n)s+3kn^2\\
		& \ \ \ +k(k+1)\delta^2+(1+2k-4kn-2k^2n)\delta-4n-5kn+1.
	\end{align*}
Let
	\begin{align*}
		h(s)= g(n-1)&=(k^2+k)s^2+(1+2k+k\delta+k^2\delta-4kn-2k^2n)s+3kn^2\\
		& \ \ \ +k(k+1)\delta^2+(1+2k-4kn-2k^2n)\delta-4n-5kn+1.
	\end{align*}
Note that $\delta<s\le \frac{n-1}{k+1}$ and $n>\frac{3+5k}{k^2}+3+(\frac{3}{k}+5+2k)\delta$, we have
	\begin{align*}
		\frac{dh}{ds}&=(2k^2+2k)s+1+2k+k\delta+k^2\delta-4kn-2k^2n\\
		&\le2k(n-1)+1+2k+k\delta+k^2\delta-4kn-2k^2n\\
		&=-2kn+1+k(k+1)\delta-2k^2n<0.
	\end{align*}
Hence $h(s)$ is decreasing with respect to $s\in[\delta+1,\frac{n-1}{k+1}]$. By a simple calculation, we have
	\begin{align*}
		h(s)\ge h(\frac{n-1}{k+1})&=\frac{1}{k+1}[k^2n^2-(3+5k+3k^2+(3k+5k^2+2k^3)\delta)n\\
		& \ \ \ +k^3\delta^2+2\delta^2k^2+k\delta^2+k^2\delta+2k\delta+1]\\
		& \ \ \ >\frac{1}{k+1}[k^2n^2-(3+5k+3k^2+(3k+5k^2+2k^3)\delta)n].
	\end{align*}
Recall that $n\ge\frac{3+5k}{k^2}+3+(\frac{3}{k}+5+2k)\delta$, we have $h(s)>0$ and $g(x)\ge g(n-1)=h(s)>0$, which implies $f_s(x)<f_{\delta}(x)$ for $x\in[n-1,+\infty)$. Since $\min\{\mu(G_1),\mu(G_2)\}>n-1$, we have $\mu(G_1)>\mu(G_2)$. Thus
	$$\mu(G)\ge\mu(G_1)>\mu(G_2)=\mu(K_{\delta} \vee (K_{n-k\delta-\delta-1}+(k\delta+1)K_1),$$
a contradiction. This completes the proof.
\end{proof}

\section*{Statement}

The article has been further improved. All of these contributions were provided by Ligong Wang. In recognition of his contributions, Ligong Wang is now acknowledged as a new co-author. 
We affirm that all authors have reviewed and approved this update.





\end{document}